\documentclass[english, 11pt]{amsart}

\usepackage{babel}
\usepackage[utf8x]{inputenc}
\usepackage{mathrsfs}
\usepackage{amsmath}
\usepackage{amssymb}
\usepackage{amsthm}
\usepackage[all, cmtip]{xy}
\usepackage{leftidx}
\usepackage{amsmath,amssymb,amsfonts}
\usepackage{amsmath,amscd}
\usepackage{textcomp}
\usepackage{pictexwd,dcpic}

\newtheorem{thm}{Theorem}[section]
\newtheorem{lemma}[thm]{Lemma}
\newtheorem{remark}[thm]{Remark}
\newtheorem{proposition}[thm]{Proposition}

\theoremstyle{definition}
\newtheorem{step}{Step}

\newcommand{\lla}[0]{{\langle\!\langle}}
\newcommand{\rra}[0]{{\rangle\!\rangle}}
\newcommand{\la}[0]{{(\!(}}
\newcommand{\ra}[0]{{)\!)}}
\newcommand{\dbar}{\overline{\partial}}
\newcommand{\ddbar}{\partial \overline{\partial}}

\newcommand\ac{{\rm ac}}

\begin{document}

\renewcommand{\subjclassname}{%
\textup{2010} Mathematics Subject Classification}

\title[A general extension theorem for cohomology classes]{A general 
extension theorem for cohomology classes \\ 
on non reduced analytic subspaces\vskip2mm\strut}

\date{\today, version 2}



\thanks{This work was supported by the Agence Nationale de la Recherche grant
``Convergence de Gromov-Hausdorff en géométrie k\"ahlérienne'' (ANR-GRACK),
the European Research Council project ``Algebraic and K\"ahler Geometry''
(ERC-ALKAGE, grant No. 670846 from September 2015), the Japan Society
for the Promotion of Science Grant-in-Aid for Young Scientists (B) (JSPS,
grant No. 25800051) and the JSPS Program for Advancing Strategic
International Networks to Accelerate the Circulation
of Talented Researchers.}

\author[Junyan Cao, Jean-Pierre Demailly, Shin-ichi Matsumura]{
Junyan Cao${}^{(1)}$, Jean-Pierre Demailly${}^{(2)}$,
Shin-ichi Matsumura${}^{(3)}$}
\address{\hskip-\parindent(1) Institut de Math\'ematiques de Jussieu,
Universit\'e Pierre et Marie Curie, 4 place Jussieu,
75252~Paris, {\it e-mail}\/:
{\tt junyan.cao@imj-prg.fr}\vskip3pt\hskip-\parindent
(2) Institut Fourier, Universit\'e Grenoble Alpes, 100 rue des Maths,
38610 Gi\`eres, France, {\it e-mail}\/:
{\tt jean-pierre.demailly@univ-grenoble-alpes.fr}\vskip3pt\hskip-\parindent
(3) Mathematical Institute, Tohoku University, 
6-3, Aramaki Aza-Aoba, Aoba-ku, Sendai 980-8578, Japan, {\it e-mail}\/:
{\tt mshinichi@m.tohoku.ac.jp, mshinichi0@gmail.com}}
\maketitle

\vskip-2.4cm
{\hfill\emph{In memory of Professor Lu QiKeng {\rm (1927--2015)}\kern7.2mm\strut}\hfill}
\vskip1.5cm\strut

\begin{abstract}
The main purpose of this paper is to generalize the celebrated $L^2$ extension theorem of Ohsawa-Takegoshi in several directions~: the holomorphic sections to extend are taken in a possibly singular hermitian line bundle, the subvariety from which the extension is performed may be non reduced, the ambient manifold is K\"ahler and holomorphically convex, but not necessarily compact.
\end{abstract}

\section{Introduction and preliminaries}

The purpose of this paper is to generalize the celebrated $L^2$ extension theorem of Ohsawa-Takegoshi \cite{OT87} under the weakest possible hypotheses, along the lines of \cite{Dem15b} and \cite{Mat16b}. Especially, the ambient complex manifold $X$ is a K\"ahler manifold that is only assumed to be \emph{holomorphically convex}, and is not necessarily compact; by the Remmert reduction theorem, this is the same as a K\"ahler manifold $X$ that admits a proper holomorphic map $\pi:X\to S$ onto a Stein complex space $S$. This allows in particular to consider relative situations over a Stein base. We consider a holomorphic line bundle $E\to X$ equipped with a singular hermitian metric $h$, namely a metric which can be expressed locally as $h=e^{-\varphi}$ where $\varphi$ is a \emph{quasi-psh} function, i.e.\ a function that is locally the sum $\varphi=\varphi_0+u$ of a plurisubharmonic function $\varphi_0$ and of 
a smooth function $u$. Such a bundle admits a curvature current
\begin{equation}
\Theta_{E, h}:=i\ddbar\varphi=i\ddbar\varphi_0+i\ddbar u
\end{equation}
which is locally the sum of a positive $(1,1)$-current 
$i\ddbar \varphi_0$ and a smooth
$(1,1)$-form $i\ddbar u$. Our goal is to extend sections that
are defined on a (non necessarily reduced) complex subspace $Y\subset X$,
when the structure sheaf 
\hbox{$\mathcal{O}_Y:=\mathcal{O}_X/\mathcal{I}(e^{-\psi})$} is given by the 
multiplier ideal sheaf of a quasi-psh function $\psi$ 
with \emph{neat analytic singularities}, i.e.\
locally on a neighborhood $V$ of an arbitrary point $x_0\in X$ we have
\begin{equation}
\psi(z)=c\log\sum|g_j(z)|^2+v(z),\qquad g_j\in\mathcal{O}_X(V),~~
v\in C^\infty(V).
\end{equation}
Let us recall that the multiplier ideal sheaf $\mathcal{I}(e^{-\varphi})$ of
a quasi-psh function $\varphi$ is defined by
\begin{equation}
\mathcal{I}(e^{-\varphi})_{x_0}=\big\{f\in\mathcal{O}_{X,x_0}\,;\;\exists U\ni x_0\,,\;
\int_U|f|^2e^{-\varphi}d\lambda<+\infty\big\}
\end{equation}
with respect to the Lebesgue measure $\lambda$ in some local coordinates 
near~$x_0$. As usual, we also denote by $K_X=\Lambda^nT^*_X$ the 
canonical bundle of an $n$-dimensional complex manifold $X$. 
As is well known,
$\mathcal{I}(e^{-\varphi})\subset\mathcal{O}_X$ is a coherent ideal sheaf
(see e.g.\ \cite{Dem-book}). Our main result is given by the following general
statement.

\begin{thm}\label{main-theorem}
Let $E$ be a holomorphic line bundle over a holomorphically
convex K\"ahler mani\-fold~$X$. Let $h$ be a possibly singular hermitian 
metric on $E$, $\psi$~a quasi-psh function with neat analytic singularities
on $X$. Assume that 
there exists a positive continuous function $\delta>0$ on $X$ such that
\begin{equation}\label{main-curv-cond}
\Theta_{E,h}+(1+\alpha\delta)i\ddbar \psi \geq 0\qquad\text{in the sense
of currents, for all}~~ 
\alpha\in[0,1].
\end{equation}
Then the morphism induced by the natural inclusion 
$\mathcal{I}(he^{-\psi}) \to \mathcal{I}(h)$ 
\begin{equation}\label{main-inj}
H^{q}(X, K_X \otimes E \otimes \mathcal{I}(he^{-\psi})) 
\to H^{q}(X, K_X \otimes E \otimes \mathcal{I}(h)) 
\end{equation}
is injective for every $q\geq 0$. 
In other words, the morphism induced by the natural sheaf surjection $\mathcal{I}(h) \to  \mathcal{I}(h)/\mathcal{I}(he^{-\psi})$
\begin{equation}\label{main-surj}
H^{q}(X, K_X \otimes E \otimes \mathcal{I}(h)) \to 
H^{q}(X, K_X \otimes E \otimes \mathcal{I}(h)/\mathcal{I}(he^{-\psi})) 
\end{equation}
is surjective for every $q\geq 0$. 
\end{thm}

\begin{remark}\label{main-consequence} {\rm
If $h$ is smooth, we have $\mathcal{I}(h)=\mathcal{O}_X$ and 
$\mathcal{I}(h)/\mathcal{I}(he^{-\psi})=\mathcal{O}_X/\mathcal{I}(e^{-\psi})
:=\mathcal{O}_Y$ where $Y$ is the zero subvariety of the ideal 
sheaf~$\mathcal{I}(e^{-\psi})$. Then for $q=0$, the surjectivity statement can
be interpreted an extension theorem for holomorphic sections,
with respect to the restriction morphism
\begin{equation}\label{main-degree-zero}
H^{0}(X, K_X \otimes E)\to 
H^{0}(Y, (K_X \otimes E)_{|Y}).
\end{equation}
In general, the quotient sheaf 
$\mathcal{I}(h)/\mathcal{I}(he^{-\psi})$ is supported in
an analytic subvariety $Y\subset X$, which is the zero set of the
quotient ideal
\begin{equation*}
\mathcal{J}_Y:=\mathcal{I}(he^{-\psi}):
\mathcal{I}(h)=\big\{f\in\mathcal{O}_X\,;\;f\cdot\mathcal{I}(h)\subset
\mathcal{I}(he^{-\psi})\big\},
\end{equation*}
and (\ref{main-surj}) can be considered as a restriction morphism to $Y$.\qed}
\end{remark}
\vskip8pt

The crucial idea of the proof is to prove the results (say, in the form of the surjectivity statement), only up to approximation. This is done by solving a $\dbar$-equation
\begin{equation*}
\dbar u_\varepsilon+w_\varepsilon=v
\end{equation*}
where the right hand side $v$ is given and $w_\varepsilon$ is an error term such
that $\Vert w_\varepsilon\Vert=O(\varepsilon^a)$ as $\varepsilon\to 0$,
for some constant $a>0$. A twisted Bochner-Kodaira-Nakano identity introduced by
Donnelly and Fefferman \cite{DF83}, and Ohsawa and Takegoshi \cite{OT87} is used for that purpose, with an additional correction term. The version 
we need can be stated as follows.

\begin{proposition}{\rm (see \cite[Prop.~3.12]{Dem15b})}\label{L2-estimate}
Let $X$ be a complete K\"ahler manifold
equipped with a $($non necessarily complete$)$ K\"ahler metric
$\omega$, and let $(E,h)$ be a Hermitian vector bundle over~$X$.
Assume that there are smooth and bounded functions $\eta,\,\lambda>0$
on $X$ such that the curvature operator
$$
B=B^{n,q}_{E,h,\omega,\eta,\lambda}=
[\eta\,\Theta_{E,h}-i\,\ddbar\eta-i\lambda^{-1}d\partial\eta\wedge\dbar\eta,
\Lambda_\omega]\in C^\infty(X,\mathop{\rm Herm}(\Lambda^{n,q}T^*_X\otimes E))
$$
satisfies $B+\varepsilon I>0$ for some $\varepsilon>0$ $($so that $B$
can be just semi-positive or even slightly negative; here $I$ is the identity 
endomorphism$)$. Given a section
$v\in L^2(X,\Lambda^{n,q}T^*_X\otimes E)$ such that $\dbar v=0$ and
$$M(\varepsilon):=
\int_X\langle (B+\varepsilon I)^{-1}v,v\rangle\,dV_{X,\omega}<+\infty,$$
there exists an approximate solution
\hbox{$f_\varepsilon\in L^2(X,\Lambda^{n,q-1}T^*_X \otimes E)$} and a correction
term $w_\varepsilon\in L^2(X,\Lambda^{n,q}T^*_X \otimes E)$ such that
$\dbar u_\varepsilon=v-w_\varepsilon$ and
$$
\int_X(\eta+\lambda)^{-1}|u_\varepsilon|^2\,dV_{X,\omega}+
\frac{1}{\varepsilon}\int_X|w_\varepsilon|^2\,dV_{X,\omega}\le M(\varepsilon).
$$
Moreover, if $v$ is smooth, then $u_\varepsilon$ and $w_\varepsilon$ can be taken smooth.
\end{proposition}

In our situation, the main part of the solution, namely $u_\varepsilon$, may very well explode as $\varepsilon\to 0$. In order to show that the equation $\dbar u=v$ can be solved, it is therefore needed to check that the space of coboundaries is closed in the space of cocycles in the Fr\'echet topology under consideration (here, the $L^2_{\rm loc}$ topology), in other words, that the related cohomology group
$H^q(X,\mathcal{F})$ is Hausdorff. In this respect, the fact of considering
$\dbar$-cohomology of smooth forms equipped with the $C^\infty$ topology
on the one hand, or cohomology of forms $u\in L^2_{\rm loc}$ with
$\dbar u\in L^2_{\rm loc}$ on the other hand, yields the same topology on
the resulting  cohomology group $H^q(X,\mathcal{F})$. This comes from the fact
that both complexes yield fine resolutions of the same coherent sheaf
$\mathcal{F}$, and the topology of $H^q(X,\mathcal{F})$ can also be obtained
by using \v{C}ech cochains with respect  to a Stein covering $\mathcal{U}$
of~$X$. The required Hausdorff property then comes from the following well
known fact.

\begin{lemma}\label{hausdorff-lemma}Let $X$ be a holomorphically convex
complex space and $\mathcal{F}$ a coherent analytic sheaf over $X$.
Then all cohomology groups $H^q(X,\mathcal{F})$ are Hausdorff with respect
to their natural topology $($induced by the Fr\'echet topology of local uniform
convergence of holomorphic 
cochains$)$.\footnote{~It was pointed out to us by Prof.\ Takeo Ohsawa that
this result does not hold under the assumption that $X$ is weakly
pseudoconvex, i.e., if we only assume that $X$ admits a smooth psh 
exhaustion. A counter-example can be derived from \cite{Kaz84}. As a 
consequence, it is unclear whether the results
of the present paper extend to the K\"ahler weakly pseudoconvex case, although
the main $L^2$ estimates are still valid in that situation.}
\end{lemma}

In fact, the Remmert reduction theorem implies that $X$ admits a proper
holomorphic map $\pi:X\to S$ onto a Stein space $S$, and Grauert's
direct image theorem shows that all direct images $R^q\pi_*\mathcal{F}$
are coherent sheaves on~$S$. Now, as $S$ is Stein, Leray's theorem combined
with Cartan's theorem B tells us that we have
an isomorphism $H^q(X,\mathcal{F})\simeq H^0(S,R^q\pi_*\mathcal{F})$. More
generally, if $U\subset S$ is a Stein open subset, we have
\begin{equation}\label{local-over-S}
H^q(\pi^{-1}(U),\mathcal{F})\simeq H^0(U,R^q\pi_*\mathcal{F})
\end{equation}
and when $U\Subset S$ is relatively compact,
it is easily seen that this a topological isomorphism of Fr\'echet spaces
since both sides are $\mathcal{O}_S(U)$ modules of finite type and can be
seen as a Fr\'echet quotient of some direct sum
$\mathcal{O}_S(U)^{\oplus N}$ by looking at local generators and local relations
of~$R^q\pi_*\mathcal{F}$.
Therefore $H^q(X,\mathcal{F})\simeq H^0(S,R^q\pi_*\mathcal{F})$
is a topological isomorphism and the space of sections in the right
hand side is a Fr\'echet space. In particular,
$H^q(X,\mathcal{F})$ is Hausdorff.\qed
\vskip8pt
The isomorphism (\ref{local-over-S}) shows that it is enough to prove
Theorem~\ref{main-theorem} locally over $X$, i.e., we can replace $X$
by $X'=\pi^{-1}(S')\Subset X$ where $S'\Subset S$. Therefore, we can
assume that $\delta>0$ is a constant rather than a continuous
function.

\section{Proof of the extension theorem}

In this section, we give a proof of Theorem~\ref{main-theorem} based on a generalization of the arguments of \cite[Th.~2.14]{Dem15b}. We start by proving the special case of the extension result for holomorphic sections ($q=0$).

\begin{thm}\label{surject-sections}
Let $(X, \omega)$ be a holomorphically convex K\"{a}hler manifold and $\psi$ be a quasi-psh function with neat analytic singularities. 
Let $E$ be a line bundle with a possibly singular metric~$h$, and
$Y$ the support of the sheaf $\mathcal{I} (h)/\mathcal{I}(h e^{-\psi})$,
along with the structure sheaf
\hbox{$\mathcal{O}_Y:=\mathcal{I}(h e^{-\psi}):\mathcal{I} (h)$}.
Assume that there is a continuous function $\delta >0$ such that
\begin{equation*}
i\Theta_{E,h} + (1+\alpha\delta) i\partial \overline{\partial}\psi \geq 0 \qquad\text{in the sense of currents, for all $\alpha\in [0, 1]$}. 
\end{equation*}
Then the restriction morphism 
$$H^0 (X, \mathcal{O}_X (K_X \otimes E)\otimes \mathcal{I} (h)) \rightarrow H^0 (Y, \mathcal{O}_X (K_X \otimes E) \otimes \mathcal{I} (h) /\mathcal{I} (h e^{-\psi})_{|Y})$$
is surjective.
\end{thm}

\begin{proof}
{\bf (a)} Let us first assume for simplicity that $h$ is smooth. We will
explain the general case later. Then $\mathcal{I}(h)=\mathcal{O}_X$ and
$\mathcal{I} (h) /\mathcal{I} (h e^{-\psi})=\mathcal{O}_Y=
\mathcal{O}_X/\mathcal{I} (e^{-\psi})$. After possibly shrinking $X$ into
a relatively compact holomorphically convex open subset $X'=\pi^{-1}(S')
\Subset X$, we can suppose that $\delta>0$ is a constant and
that $\psi\leq 0$, after subtracting a large constant to~$\psi$. 
Also, without loss of generality, we can assume that $\psi$ admits 
a discrete sequence of ``jumping numbers''
\begin{equation}\label{jumping-numbers}
0=m_0<m_1<\cdots <m_p<\cdots\quad\hbox{
such that $\mathcal{I}(m\psi)=\mathcal{I}(m_p\psi)$ for
$m\in[m_p,m_{p+1}[$}.
\end{equation}
 Since $\psi$ is assumed to have analytic singularities, 
this follows from using a log resolution of singularities, thanks to the
Hironaka desingularization theorem (by the much deeper result of
\cite{GZ15} on the strong openness conjecture, one could even possibly 
eliminate the assumption that $\psi$ has analytic singularities).
We fix here $p$ such that $m_p\leq 1<m_{p+1}$, and
in the notation of \cite{Dem15b}, we let $Y=Y^{(m_p)}$ be defined by the
non necessarily reduced
structure sheaf $\mathcal{O}_Y=\mathcal{O}_X/\mathcal{I} (e^{-\psi})=
\mathcal{O}_X/\mathcal{I} (e^{-m_p\psi})$.
\vskip8pt

\noindent
\begin{step}[Construction of a smooth extension]\label{Step1}
Take
$$
f \in H^0 (Y, \mathcal{O}_X (K_X \otimes E)_{|Y})=
H^0 (X, \mathcal{O}_X (K_X \otimes E) \otimes \mathcal{O}_X /
\mathcal{I} (e^{-m_p\psi})).
$$
Let $\mathcal{U}=(U_i)$ be a Stein covering of $X$ and let $(\rho_i)$
be a partition of unity subordinate to $(U_i)$. 
Thanks to the exact sequence 
\begin{equation}\label{exact-sequence}
0\rightarrow \mathcal{I} (e^{-\psi}) \rightarrow \mathcal{O}_X \rightarrow \mathcal{O}_X /\mathcal{I} (e^{-\psi})\rightarrow 0 , 
\end{equation}
we can find a $\widetilde{f}_i \in H^0 (U_i, \mathcal{O}_X (K_X \otimes E))$ such that
\begin{equation*}
\widetilde{f}_i |_{Y \cap U_i} = f |_{Y \cap U_i} .
\end{equation*}
Then \eqref{exact-sequence} implies that
\begin{equation}\label{f-tilde-differences}
\widetilde{f}_i -\widetilde{f}_j \in H^0 (U_i \cap U_j , \mathcal{O}_X (K_X \otimes E) \otimes \mathcal{I} (e^{-\psi})) .
\end{equation}
As a consequence, the smooth section $\widetilde{f} := \sum\limits_i
\rho_i \cdot \widetilde{f}_i$ is a smooth extension of $f$ and satisfies
$\dbar\widetilde{f} = \sum\limits_i (\dbar\rho_i)\cdot (\widetilde{f}_i-
\widetilde{f}_j)$ on $U_j$, hence
\begin{equation}\label{dbar-f-bound}
\int_X |\dbar\widetilde{f} |_{\omega, h} ^2 e^{-\psi} dV_{X,\omega} =
\int_X \sum_j\rho_j\Big|\sum_i (\overline{\partial}\rho_i) \cdot 
(\widetilde{f}_i - \widetilde{f}_j)\Big|_{\omega, h} ^2 e^{-\psi}
dV_{X, \omega} < +\infty . 
\end{equation}
\vskip2pt
\end{step}

\begin{step}[$L^2$-estimates]\label{Step2}
We follow here the arguments of \cite[proof of th.~2.14, p.~217]{Dem15b}. Let $t\in \mathbb{Z}^-$ and let $\chi_t$ be the negative convex increasing function defined in \cite[(5.8$*$), p.~211]{Dem15b}. Put $\eta_t := 1 -\delta \cdot \chi_t(\psi)$
and $\lambda_t := 2 \delta \frac{(\chi_t^2(\psi))^2}{\chi_t''(\psi)}$. We set
\begin{eqnarray*}
R_t &:=& \eta_t (\Theta_{E,h} +i\ddbar\psi ) -i\ddbar \eta_t -\lambda_t ^{-1} i \partial \eta_t \wedge \dbar \eta_t\\
&\kern3pt =& \eta_t (\Theta_{E,h} + (1+\delta \eta_t ^{-1} \chi_t'(\psi))i\ddbar\psi ) 
+ \frac{\delta\cdot \chi_t''(\psi)}{2}  i \partial\psi\wedge \dbar\psi.
\end{eqnarray*}
Note that $\chi_t''(\psi) \geq \frac{1}{8}$ on $W_t =\{t < \psi < t+1\}$. The curvature assumption \eqref{main-curv-cond} implies
$$\Theta_{E,h} + (1+\delta \eta_t ^{-1} \chi_t'(\psi))\,i\ddbar\psi \geq 0 \qquad\text{on }X .$$
As in \cite{Dem15b}, we find
\begin{equation}\label{posit-curv-bound1}
R_t \geq 0 \qquad\text{on~~}X
\end{equation}
and
\begin{equation}\label{posit-curv-bound2}
R_t \geq \frac{\delta}{16} i \partial\psi\wedge \dbar\psi\qquad\text{on~~}W_t =\{t < \psi < t+1\}. 
\end{equation}
\vskip8pt

\noindent
Let $\theta : [ -\infty , +\infty [{}\rightarrow [0,1]$ be a smooth non increasing real function satisfying
$\theta (x)=1$ for $x \leq 0$, $\theta (x)=0$ for $x \geq 1$ and $|\theta' | \leq 2$.
By applying the $L^2$ estimate (Proposition \ref{L2-estimate}),
for every $\varepsilon >0$ we can find  sections 
$u_{t, \varepsilon}$, $w_{t, \varepsilon}$ satisfying 
\begin{equation}\label{l2estimate}
\dbar u_{t, \varepsilon} + w_{t, \varepsilon}=v_t:=
\dbar\big(\theta(\psi -t) \cdot \widetilde{f}\;\big)
\end{equation}
and
\begin{equation}\label{main-l2-estimate}
\int_X (\eta_t+\lambda_t)^{-1}| u_{t, \varepsilon}|_{\omega, h} ^2 e^{-\psi} dV_{X,\omega}+
\frac{1}{\varepsilon} \int_X |w_{t, \varepsilon}|_{\omega, h}^2 e^{-\psi} 
dV_{X,\omega}\leq 
\int_X \langle (R_t +\varepsilon I)^{-1}v_t,v_t\rangle e^{-\psi}dV_{X,\omega}\;,
\end{equation}
where
\begin{equation}\label{vt-terms}
v_t=\dbar\big(\theta(\psi -t) \widetilde{f}\;\big)=
\theta'(\psi -t)\,\dbar\psi\wedge\widetilde{f} +
\theta(\psi -t)\,\dbar\widetilde{f}.
\end{equation}
Combining \eqref{posit-curv-bound1}, \eqref{posit-curv-bound2},
\eqref{main-l2-estimate} and \eqref{vt-terms}, we get
$\int_X | u_{t, \varepsilon}|_{\omega,h}^2 e^{-\psi} dV_{X,\omega}< +\infty$ and
\begin{equation}\label{error-term}
 \int_X |w_{t, \varepsilon}|_{\omega, h} ^2 e^{-\psi} d V_{X, \omega} \leq 
\frac{128\,\varepsilon}{\delta} \int_{\{t < \psi < t+1\}} |\widetilde{f}\;|_{\omega,h} ^2 e^{-\psi} 
dV_{X,\omega}+ 2 \int_{\{\psi < t+1\}}
|\dbar\widetilde{f}\;|_{\omega, h} ^2 e^{-\psi} d V_{X, \omega}. 
\end{equation}
\vskip8pt

\noindent
We now estimate the right hand side of \eqref{error-term}.
Since $\widetilde{f}$ is smooth, we have an obvious upper bound of the 
first term
\begin{equation}\label{second-term}
 \int_{\{t < \psi < t+1\}} |\widetilde{f}|_{\omega,h} ^2 e^{-\psi} d V_{X,\omega}\leq C_1 e^{-t} ,
\end{equation}
where $C_1$ is the $C^0$ norm of $\widetilde{f}$.
For the second term,
thanks to (\ref{jumping-numbers}), (\ref{f-tilde-differences}) and
(\ref{dbar-f-bound}), we have
\begin{equation}\label{dbar-f-bound-stronger}
\int_X |\dbar\widetilde{f} |_{\omega, h} ^2 e^{-(1+\alpha)\psi}  d V_{X, \omega}< +\infty
\end{equation}
for any $\alpha\in{}]0,m_{p+1}-1[$. As a consequence, we get
\begin{equation}\label{first-term}
 \int_{\{\psi < t+1\}} |\dbar\widetilde{f} |_{\omega, h}^2 e^{-\psi} d V_{X, \omega} \leq C_2  e^{\alpha t}
\end{equation}
for some constant $C_2$ depending only on $\alpha$.
By taking $\varepsilon =e^{(1+\alpha)t}$, \eqref{error-term}, \eqref{second-term} and \eqref{first-term}  imply
\begin{equation}\label{addined}
 \int_X |w_{t, \varepsilon}|_{\omega, h}^2 e^{-\psi} d V_{X, \omega}\leq C_3 e^{\alpha t} =O(\varepsilon^{\frac{\alpha}{1+\alpha}}), 
\end{equation}
for some constant $C_3$, whence the error tends to $0$ as $t\to-\infty$ and
$\varepsilon\to 0$.
\vskip8pt
\end{step}

\begin{step}[Final conclusion]\label{Step3}
Putting everything together and redefining
$u_t=u_{t,\varepsilon}$, $w_t=w_{t,\varepsilon}$ for simplicity of notation,
we get
\begin{equation}\label{approximationsolution}
  \dbar(\theta(\psi -t) \cdot \widetilde{f}  -u_t) =
  w_t ,\qquad  \int_X |u_t|_h ^2 e^{-\psi}
  dV_{X,\omega}< +\infty  
\end{equation}
and
\begin{equation}\label{l2-decay}
  \lim\limits_{t \rightarrow -\infty} \int_X |w_t|_{\omega, h}^2
  e^{-\psi} dV_{X,\omega}=0 . 
\end{equation}
After shrinking $X$, we can assume that we have a finite Stein covering $\mathcal{U}=(U_i)$ where the $U_i$ are biholomorphic to bounded pseudoconvex domains. The standard H\"ormander
$L^2$ estimates then provide $L^2$ sections $s_{t,j}$ on $U_j$
such that $\dbar s_{t,j}=w_t$ on $U_j$ and
\begin{equation}\label{l2-decay2}
  \lim\limits_{t \rightarrow -\infty} \int_{U_j}|s_{t,j}|_{\omega, h}^2
  e^{-\psi} dV_{X,\omega}=0 . 
\end{equation}
Then
\begin{eqnarray}
\dbar\Big(\theta(\psi -t) \cdot \widetilde{f}  -u_t
-\sum_j\rho_js_{t,j}\Big)
&=&-\sum_j (\dbar\rho_j)\cdot s_{t,j}\quad\hbox{on $X$}\nonumber\\
\label{l2-decay3}&=&-\sum_j (\dbar\rho_j)\cdot (s_{t,j}-s_{t,i})\quad
\hbox{on $U_i$}.
\end{eqnarray}
As $\dbar(s_{t,j}-s_{t,i})=0$ on $U_i\cap U_j$, the difference is holomorphic
and the right hand side of \eqref{l2-decay3} is smooth. Moreover, \eqref{l2-decay2}
shows that these differences converge uniformly to $0$, hence
the right hand side of \eqref{l2-decay3} 
converges to $0$ in $C^\infty$ topology. The left hand side implies that
this is a coboundary in the $C^\infty$ Dolbeault resolution of
$\mathcal{O}_X(K_X\otimes E)$. By applying
Lemma~\ref{hausdorff-lemma}, we conclude
that there is a $C^\infty$ section $\sigma_t$ of $K_X\otimes E$ converging
uniformly to $0$ on compact subsets of~$X$ as $t\to-\infty$, such that
$\dbar\sigma_t=\sum\limits_j (\dbar\rho_j)\cdot s_{t,j}$ on $X$. This implies that
$$
\widetilde{f}_t:=\theta(\psi -t) \cdot \widetilde{f}  -u_t
-\sum_j\rho_js_{t,j}+\sigma_t
$$
is holomorphic on $X$. H\"ormander's $L^2$ estimates also produce
local smooth solutions $\sigma_{t,i}$ on $U_i$ with the
additional property that
$\lim\limits_{t\to-\infty}\int_{U_i}|\sigma_{t,i}|_{\omega,h} ^2 e^{-\psi}dV_{X,\omega}=0$.
Therefore
$$
\widetilde{f}_{t,i}:=\theta(\psi -t) \cdot \widetilde{f}  -u_t
-\sum_j\rho_js_{t,j}+\sigma_{t,i}
$$
is holomorphic on $U_i$ and $\widetilde{f}_t-\widetilde{f}_{t,i}$
converges uniformly to $0$ on compact subsets of $U_i$. However, by
construction, $\widetilde{f}_{t,i}-\widetilde{f}_i$ is a holomorphic
section on $U_i$ that satisfies the $L^2$ estimate with respect to
the weight $e^{-\psi}$, hence $\widetilde{f}_{t,i}-\widetilde{f}_i$
is a section of $\mathcal{O}_X(K_X\otimes E)\otimes\mathcal{I}(e^{-\psi})$
on $U_i$, in other words the image of $\widetilde{f}_{t,i}$ in
$$
H^0(U_i,\mathcal{O}_X(K_X\otimes E)\otimes
\mathcal{O}_X/\mathcal{I}(e^{-\psi}))
$$
coincides with $f_{|U_i}$. As a consequence, the image of $\widetilde{f}_t$
in
$$
H^0(X,\mathcal{O}_X(K_X\otimes E)\otimes
\mathcal{O}_X/\mathcal{I}(e^{-\psi}))=
H^0(Y,(K_X\otimes E)_{|Y})
$$
converges to~$f$. By the direct image argument used in the preliminary section, this density property implies the surjectivity of the restriction morphism to~$Y$.\vskip8pt
\end{step}

\noindent
{\bf(b)} We now prove the theorem for the general case when $h=e^{-\varphi}$ is not necessarily smooth. We can reduce ourselves to the case when $\psi$ has divisorial singularities (see \cite{Dem15b} or the next section for a more detailed argument). Let us pick a section
$$
f\in  H^0 (X, \mathcal{O}_X (K_X \otimes E) \otimes \mathcal{I} (h) /\mathcal{I} (h e^{- \psi})).
$$
By using the same reasoning as in Step \ref{Step1}, we can find a
smooth extension $\widetilde{f}\in \mathcal{C}^\infty (X, K_X \otimes E)$ of $f$ such that
\begin{equation}\label{dbar-f-bound-new}
\qquad \int_X |\dbar\widetilde{f}|_{\omega, h} ^2 e^{-\psi} dV_{X,\omega} < +\infty . 
\end{equation}
For every $t\in \mathbb{Z}^{-}$ fixed, as 
$\psi$ has divisorial singularities, we still have
$$\Theta_{E,h}+ (1+\delta \eta_t ^{-1} \chi_t'(\psi)) 
(i\ddbar\psi)_{\ac} \geq 0 \qquad\text{on }X ,$$
where $(i\ddbar\psi)_{\ac}$ is the absolutely continuous part of
$i\ddbar\psi$.  The regularization techniques of \cite{DPS01} and
\cite[Th.~1.7, Remark~1.11]{Dem15a} (cf.\ also the next section)
produce a family of singular
metrics $\{h_{t,\varepsilon}\}_{k=1}^{+\infty}$ which are smooth in
the complement $X\smallsetminus Z_{t,\varepsilon}$ of an analytic set,
such that $\mathcal{I} (h_{t,\varepsilon}) =\mathcal{I} (h)$,
$\mathcal{I} (h_{t,\varepsilon}e^{-\psi}) =\mathcal{I} (h e^{-\psi})$ and
$$
\Theta_{E,h_{t,\varepsilon}}+
(1+\delta \eta_t ^{-1} \chi_t'(\psi))\,i\ddbar\psi \geq
-\frac{1}{2}\varepsilon \omega \qquad\text{on }X .
$$
The additional error term $-\frac{1}{2}\varepsilon \omega$ is
irrelevant when we use Proposition \ref{L2-estimate}, as it is absorbed by
taking the hermitian operator $B+\varepsilon I$. Therefore for every
$t\in \mathbb{Z}^-$, with the adjustment $\varepsilon=e^{\alpha t}$,
$\alpha\in{}]0,m_{k+1}-1[$, we can find a singular metric
$h_t=h_{t,\varepsilon}$ which is smooth in the complement $X\setminus Z_t$
of an analytic set, such that
$\mathcal{I} (h_t) =\mathcal{I} (h)$,
$\mathcal{I} (h_te^{-\psi}) =\mathcal{I} (he^{-\psi})$
and $h_t\uparrow h$ as
$t\rightarrow -\infty$, and approximate solutions of the
$\dbar$-equation such that
$$\dbar(\theta(\psi -t) \cdot \widetilde{f}  -u_t) =w_t \text{ }, \qquad \int_X |u_t|_{\omega,h_t} ^2 e^{-\psi} dV_{X, \omega} < +\infty$$
and
$$ \lim_{t\rightarrow -\infty} \int_X |w_t|_{\omega, h_t} ^2 e^{-\psi} d V_{X,\omega} =0 .$$
Proposition \ref{L2-estimate} can indeed be applied since $X\smallsetminus Z_t$ is complete K\"ahler (at least after we shrink $X$ a little bit as $X'=\pi^{-1}(S')$,
cf.~\cite{Dem82}).
The theorem is then proved by using the same argument as in Step \ref{Step3}; it is enough to notice that the holomorphic sections $s_{t,j}-s_{t,i}$ and
$\widetilde{f}_{t,i}-\widetilde{f}_{i}$ satisfy the $L^2$-estimate with respect
to $(h_t,\psi)$ [instead of the expected $(h,\psi)$], but the multiplier ideal
sheaves involved are unchanged. The Hausdorff property is applied to
the cohomology group $H^1(X,\mathcal{O}_X(K_X\otimes E)\otimes\mathcal{I}(h))$
instead of $H^1(X,K_X\otimes E)$, and the density property to the morphism 
of direct image sheaves
$$\pi_*\big(\mathcal{O}_X(K_X\otimes E)\otimes\mathcal{I}(h)\big)\to
\pi_*\big(\mathcal{O}_X(K_X\otimes E)\otimes\mathcal{I}(h)/
\mathcal{I}(he^{-\psi})\big)
$$
over the Stein space~$S$.
\end{proof}
\vskip8pt

\noindent
{\bf Proof of the extension theorem for degree \textit{\textbf q} cohomology classes.}
The reasoning is extremely similar, so we only explain the few additional arguments needed. In fact, Proposition \ref{L2-estimate} can be applied right away to arbitrary $(n,q)$-forms with $q\geq 1$, and the twisted Bochner-Kodaira-Nakano inequality yields exactly the same estimates. Any cohomology class in
$$
H^q(Y,\mathcal{O}_X(K_X\otimes E)\otimes\mathcal{I}(h)/\mathcal{I}(he^{-\psi}))
$$
is represented by a holomorphic \v{C}ech $q$-cocycle with respect to
the Stein covering $\mathcal{U}=(U_i)$, say
$$
(c_{i_0\ldots i_q}),\qquad
c_{i_0\ldots i_q}\in H^0\big(U_{i_0}\cap\ldots\cap U_{i_q},
\mathcal{O}_X(K_X\otimes E)\otimes\mathcal{I}(h)/\mathcal{I}(he^{-\psi})\big).
$$
By the standard sheaf theoretic isomorphisms with Dolbeault cohomology (cf.\
e.g.\ \cite{Dem-e-book}), this 
class is represented by a smooth $(n,q)$-form
$$
f=\sum_{i_0,\ldots,i_q}c_{i_0\ldots i_q}\rho_{i_0}
\dbar\rho_{i_1}\wedge\ldots\dbar\rho_{i_q}
$$
by means of a partition of unity $(\rho_i)$ subordinate to $(U_i)$. This form is
to be interpreted as a form on the (non reduced) analytic subvariety $Y$ associated with the ideal sheaf $\mathcal{J}=\mathcal{I}(he^{-\psi}):
\mathcal{I}(h)$ and the structure sheaf $\mathcal{O}_Y=
\mathcal{O}_X/\mathcal{J}$. We get an extension as a smooth (no longer
$\dbar$-closed) $(n,q)$-form on $X$ by taking
$$
\widetilde{f}=\sum_{i_0,\ldots,i_q}\widetilde{c}_{i_0\ldots i_q}\rho_{i_0}
\dbar\rho_{i_1}\wedge\ldots\dbar\rho_{i_q}
$$
where $\widetilde{c}_{i_0\ldots i_q}$ is an extension of
$c_{i_0\ldots i_q}$ from $U_{i_0}\cap\ldots\cap U_{i_q}\cap Y$ to
$U_{i_0}\cap\ldots\cap U_{i_q}$. Again, we can find approximate 
$L^2$ solutions of the $\dbar$-equation such that
$$\dbar(\theta(\psi -t) \cdot \widetilde{f}  -u_t) =w_t \text{ }, \qquad \int_X |u_t|_{\omega,h_t} ^2 e^{-\psi} dV_{X, \omega} < +\infty$$
and
$$ \lim_{t\rightarrow -\infty} \int_X |w_t|_{\omega, h_t} ^2 e^{-\psi} d V_{X,\omega} =0 .$$
The difficulty is that $L^2$ sections cannot be restricted in a continuous
way to a subvariety. In order to overcome this problem,
we play again the game of returning to \v{C}ech cohomology by solving 
inductively $\dbar$-equations for $w_t$ on $U_{i_0}\cap\ldots\cap U_{i_k}$,
until we reach an equality
\begin{equation}\label{l2-estimate-nq-forms}
\dbar\big(\theta(\psi -t) \cdot \widetilde{f}  -\widetilde{u}_t\big)
=\widetilde{w}_t:=
-\sum_{i_0,\ldots,i_{q-1}}s_{t,i_0\ldots i_q}\dbar\rho_{i_0}\wedge
\dbar\rho_{i_1}\wedge\ldots\dbar\rho_{i_q}
\end{equation}
with holomorphic sections 
$s_{t,I}=s_{t,i_0\ldots i_q}$ on $U_I=U_{i_0}\cap\ldots\cap U_{i_q}$,
such that
$$ \lim_{t\rightarrow -\infty} \int_{U_I} |s_{t,I}|_{\omega, h_t} ^2 
e^{-\psi} d V_{X,\omega} =0.$$
Then the right hand side of (\ref{l2-estimate-nq-forms}) is smooth, and more
precisely has coefficients in the sheaf
$\mathcal{C}^\infty\otimes_{\mathcal{O}}\mathcal{I}(he^{-\psi})$, and 
$\widetilde{w}_t\to 0$ in $C^\infty$ topology. A~priori,
$\widetilde{u}_t$ is an $L^2$ $(n,q)$-form equal to $u_t$ plus a combination
$\sum\rho_is_{t,i}$ of the local solutions of $\dbar s_{t,i}=w_t$, plus
$\sum\rho_is_{t,i,j}\wedge\dbar\rho_j$ where
$\dbar s_{t,i,j}=s_{t,j}-s_{t,i}$, plus etc~$\ldots$~, and is such that
$$\int_{X} |\widetilde{u}_t|_{\omega, h_t}^2 
e^{-\psi} d V_{X,\omega} <+\infty.$$
Since $H^q(X,\mathcal{O}_X(K_X\otimes E)\otimes\mathcal{I}(he^{-\psi}))$ can be 
computed with the $L^2_{\rm loc}$ resolution of the coherent sheaf, or 
alternatively  with the $\dbar$-complex of $(n,{\scriptstyle\bullet})$-forms 
with coefficients in
$\mathcal{C}^\infty\otimes_{\mathcal{O}}\mathcal{I}(he^{-\psi})$, we may assume that
\hbox{$\widetilde{u}_t\in\mathcal{C}^\infty\otimes_{\mathcal{O}}
\mathcal{I}(he^{-\psi})$},
after playing again with \v{C}ech cohomology. Lemma~\ref{hausdorff-lemma}
yields a sequence of smooth $(n,q)$-forms $\sigma_t$ with coefficients in 
$\mathcal{C}^\infty\otimes_{\mathcal{O}}\mathcal{I}(h)$,
such that $\dbar\sigma_t=\widetilde{w}_t$ and $\sigma_t\to 0$ in
$C^\infty$-topology. Then $\widetilde{f}_t=
\theta(\psi -t) \cdot \widetilde{f}  -\widetilde{u}_t-\sigma_t$
is a $\dbar$-closed $(n,q)$-form on $X$ with values in
$\mathcal{C}^\infty\otimes_{\mathcal{O}}\mathcal{I}(h)\otimes
\mathcal{O}_X(E)$, whose image in
$H^q(X,\mathcal{O}_X(K_X\otimes E)\otimes\mathcal{I}(h)/\mathcal{I}(he^{-\psi}))$
converges to $\{f\}$ in $C^\infty$ Fr\'echet topology. We conclude by a
density argument on the Stein space $S$, by looking at the coherent 
sheaf morphism
$$R^q\pi_*\big(\mathcal{O}_X(K_X\otimes E)\otimes\mathcal{I}(h)\big)\to
R^q\pi_*\big(\mathcal{O}_X(K_X\otimes E)\otimes\mathcal{I}(h)/
\mathcal{I}(he^{-\psi})\big).\eqno\qed
$$

\section{An alternative proof based on injectivity theorems}\label{Sec-3}

We give here an alternative proof based on injectivity theorems, in the case when $X$ is compact K\"ahler. The case of a holomorphically convex manifold is entirely similar, so we will content ourselves to indicate the required additional arguments at the end.

\begin{proof}[Proof of Theorem \ref{main-theorem}.]
First of all, we reduce the proof of Theorem \ref{main-theorem} 
to the case when $\psi$ has divisorial singularities. 
Since $\psi$ has analytic singularities, 
there exists a modification $\pi \colon X'\to X$ such that the pull-back 
$\pi^{*}\psi$ has divisorial singularities. 
For the singular hermitian line bundle $(E', h'):=(\pi^{*}E, \pi^{*}h)$ 
and the quasi-psh function $\psi':=\pi^{*}\psi$, 
we can easily check that \begin{align*}
\pi_{*}(K_{X'} \otimes E' \otimes \mathcal{I}( h'e^{-\psi'})) &= 
K_{X} \otimes E \otimes \mathcal{I}( he^{-\psi}), \\ 
\pi_{*}(K_{X'} \otimes E' \otimes \mathcal{I}( h')) &= 
K_{X} \otimes E \otimes \mathcal{I}(h). 
\end{align*}
Hence we obtain the following commutative diagram\,$:$ 
\[\xymatrix{
H^{q}(X,K_{X} \otimes E \otimes \mathcal{I}( he^{-\psi})) \ar[d]_{\cong}^{\pi^*} \ar[r]^f \ar@{}[dr]|\circlearrowleft & H^{q}(X,K_{X} \otimes E \otimes \mathcal{I}( h)) \ar[d]^{\pi^*} \\
H^{q}(X',K_{X'} \otimes E' \otimes \mathcal{I}( h'e^{-\psi'})) \ar[r]^g & 
H^{q}(X',K_{X'} \otimes E' \otimes \mathcal{I}( h')), 
}\]
where $f$, $g$ are the morphisms induced by the natural inclusions 
and $\pi^{*}$ is the natural edge morphism. 
It follows that the left edge morphism $\pi^{*}$ is an isomorphism 
since the curvature of the singular hermitian metric $h'e^{-\psi'}$ on $E'$ 
is semi-positive by the assumption. 
Indeed, even if $h'$ does not have analytic singularities, 
we can see that  
$$
R^q\pi_{*}(K_{X'} \otimes E' \otimes \mathcal{I}( h'e^{-\psi'})) =0 \text{ for every }q>0 
$$
by \cite[Corollary 1.5]{Mat16b}. 
(In the case of $X$ being a projective variety, 
a relatively easy proof can be found in \cite{FM16}.)  
If Theorem \ref{main-theorem} can be proven when $\psi$ has divisorial singularities, 
it follows that the morphism $g$ in the above diagram is injective 
since $(E', h')=(\pi^{*}E, \pi^{*}h)$ and $\psi'=\pi^{*}\psi$ satisfy 
the assumptions in Theorem \ref{main-theorem} and $\psi'$ has divisorial singularities. 
Therefore the morphism $f$ is also injective by the commutative diagram.  

\vspace{0.1cm}
\par

Now we explain the idea of the proof of Theorem \ref{main-theorem}. 
If we can obtain equisingular approximations $h_{\varepsilon}$ of $h$ 
satisfying the following properties\,: 
\begin{align*}
\Theta_{E,h_\varepsilon}+ i\ddbar \psi \geq -\varepsilon\omega~~~\text{and}~~~
\Theta_{E,h_\varepsilon}+(1+\delta)  i\ddbar \psi \geq -\varepsilon\omega, 
\end{align*}
then a proof similar to \cite{FM16} works,  
where $\omega$ is a fixed K\"ahler form on $X$. 
In the case when $\psi$ has divisorial singularities, 
we can attain either of the above curvature properties, 
but we do not know whether we can attain them at the same time. 
For this reason, 
we will look for an essential curvature condition 
arising from the assumptions on the curvatures in Theorem \ref{main-theorem}, 
in order to use the \lq \lq twisted" Bochner-Kodaira-Nakano identity.

\vspace{0.1cm}
\par
From now on, we consider a quasi-psh $\psi$ with divisorial singularities. 
Then there exist an effective $\mathbb{R}$-divisor $D$
and a smooth $(1,1)$-form $\gamma$ on $X$ 
such that 
\begin{align*}
\frac{i}{2\pi} \ddbar \psi = [D] + \frac{1}{2\pi}\gamma
\end{align*}
in the sense of $(1,1)$-currents, 
where $[D]$ denotes the current of the integration over $D$. 
For the irreducible decomposition $D=\sum_{i=1}^{N}a_{i}D_{i}$ and 
the defining section $t_{i}$ of $D_{i}$, 
we can take a smooth hermitian metric $b_i$ on $D_i$ such that 
\begin{align*}
e^{\psi}=|s|_{b}^2:=
|t_{1}|_{b_{1}}^{2a_{1}}|t_{2}|_{b_{2}}^{2a_{2}}\cdots|t_{N}|_{b_{N}}^{2a_{N}}
~~~\text{and}~~~
-\gamma=\Theta_{b}(D):=\sum_{i=1}^{N}a_{i}\Theta_{b_{i}}(D_{i}). 
\end{align*}
For a positive number $0<c \ll 1$, 
we define the continuous functions $\sigma$ and $\eta$ on $X$ by 
\begin{align*}
\sigma= \sigma_{c}:= \log(|s|_{b}^2 + c)~~~\text{and}~~~
\eta=\eta_{c}
:=\frac{1}{c}-\chi(\sigma), 
\end{align*}
where $\chi(t):=t - \log(-t)$. 

\begin{remark}\label{rem-sigma}{\rm 
(i) We may assume that $|s|_{b}^2< 1/5$ by subtracting a positive constant 
from $\psi$. 
Further we may assume that $\sigma < \log (1/5)$ and 
$1<\chi'(\sigma)<7/4$ 
by choosing a sufficiently small $c>0$.  
\vspace{0.1cm}\\
(ii) 
Further, the function $\eta$ is a continuous function on $X$ with
$\eta > 1/c$. The function $\eta$  is smooth on $X \setminus D$, 
but it need not be smooth on $X$ since $|s|_{b}^2$ is not smooth 
in the case when $0 < a_i < 1$ for some $i$. }
\end{remark}

Throughout the proof, we fix a K\"ahler form $\omega$ on $X$. 
The following proposition gives 
a suitable approximation of a singular hermitian metric $h$ on $E$, 
which enables us to use the twisted Bochner-Kodaira-Nakano identity. 
The proof is based on the argument in \cite{Ohs04}, \cite{Fuj13} 
and  the equisingular approximation theorem in \cite[Theorem 2.3]{DPS01}.  

\begin{proposition}\label{appro}
There exist singular hermitian metrics $\{h_{\varepsilon}\}_{0<\varepsilon\ll 1}$ on $E$ 
with the following properties\,$:$
\begin{itemize}
\item[(a)] $h_{\varepsilon}$ is smooth on $X \setminus Z_{\varepsilon}$, 
where $Z_{\varepsilon}$ is a proper subvariety on $X$. 
\item[(b)] $h_{\varepsilon'} \leq h_{\varepsilon''} \leq h$ holds on $X$ 
for $\varepsilon' > \varepsilon'' > 0$.
\item[(c)] $\mathcal{I}(h)= \mathcal{I}(h_{\varepsilon})$ and $\mathcal{I}(he^{-\psi})= \mathcal{I}(h_{\varepsilon}e^{-\psi})$ on $X$.
\item[(d)] $\eta (\Theta_{h_{\varepsilon}}(E)+\gamma) -  i\ddbar \eta 
- \eta^{-2} i \partial \eta \wedge \dbar \eta \geq -\varepsilon\omega$ on $X \setminus D$. 
\item[(e)] For arbitrary $t>0$, by taking a sufficiently small $\varepsilon>0$, we have 
$$
\int e^{-t \phi} - e^{-t \phi_{\varepsilon}} < \infty, 
$$
where $\phi$ $($resp. $\phi_{\varepsilon}$$)$ is a local weight of $h$ $($resp. $h_{\varepsilon}$$)$. 
\end{itemize}
\end{proposition}
\begin{proof}
We fix a a sufficiently small $c$ with $7c/4 \leq \delta$. 
Then, by Remark \ref{rem-sigma}, 
we can easily check that 
$$
\frac{ \chi'(\sigma) |s|^2_{b}}{\eta (|s|^2_{b} +c)} \leq 
\frac{7 }{4\eta} \leq  \frac{7}{4}c \leq \delta. 
$$
In particular, it follows that 
$$
\Theta_{h} + 
\Big( 1+ \frac{ \chi'(\sigma) |s|^2_{b}}{\eta (|s|^2_{b} +c)} \Big) \gamma \geq 0
\quad
\text{ on } X
$$
since $\psi$ has divisorial singularities and satisfies the
assumptions in Theorem \ref{main-theorem}.  By applying the
equisingular approximation theorem (\cite[Theorem 2.3]{DPS01}) to $h$,
we can take singular hermitian metrics
$\{h_{\varepsilon}\}_{0<\varepsilon\ll 1}$ on $E$ satisfying
properties (a), (b), (e), the former conclusion of~(c), and the
following curvature property:
$$
\Theta_{h_{\varepsilon}} + 
\Big( 1+ \frac{\chi'(\sigma) |s|^2_{b}}{\eta (|s|^2_{b} +c)} \Big) \gamma \geq - \varepsilon\omega\quad
\text{ on } X. 
$$

Now we check property~(d) from the above curvature property. 
The function $\eta$ may not be smooth on $X$, 
but it is smooth on $X \setminus D$. 
Therefore the same computation as in \cite{Ohs04} and \cite{Fuj13} works on $X \setminus D$. 
In particular, from a complicated but straightforward computation, 
we obtain 
\begin{align*}
-  i\ddbar \eta = 
- \frac{\chi'(\sigma) |s|^2_{b}}{|s|^2_{b} + c} \Theta_{b}(D)
+ \Big( \frac{c}{ \chi'(\sigma) |s|^2_{b}} + \frac{\chi''(\sigma)}{\chi'(\sigma)^2} 
\Big) i \partial \eta \wedge \dbar \eta 
\end{align*}
on $X \setminus D$ 
(see \cite{Fuj13} for the precise computation).  
Then, by $- \gamma = \Theta_{b}(D)$ on $X \setminus D$, 
we can see that 
\begin{align*}
&\eta (\Theta_{h_{\varepsilon}}(E)+\gamma) -  i\ddbar \eta 
- \frac{1}{\eta^{2}} i \partial \eta \wedge \dbar \eta \\
=& \Big( \frac{c}{ \chi'(\sigma) |s|^2_{b}} + \frac{\chi''(\sigma)}{\chi'(\sigma)^2} 
- \frac{1}{\eta^{2}}  \Big) i \partial \eta \wedge \dbar \eta  
+
\eta \Big( \Theta_{h_{\varepsilon}}(E)+
(1 + \frac{\chi'(\sigma) |s|^2_{b}}{\eta(|s|^2_{b} + c)} )  \gamma \Big) \\
\geq &\Big( \frac{c}{ \chi'(\sigma) |s|^2_{b}} + \frac{\chi''(\sigma)}{\chi'(\sigma)^2} 
- \frac{1}{\eta^{2}}  \Big) i \partial \eta \wedge \dbar \eta  
- \varepsilon\eta  \omega 
\end{align*}
on $X \setminus D$. 
A straightforward computation yields that 
$ \chi''(\sigma) / \chi'(\sigma)^2 \geq 1/ \eta^{2}$, 
and thus 
the first term is semi-positive. Since $\eta$ is bounded above,
we infer that property~(d) holds.

Finally we check the last conclusion of property~(c) by proving the following lemma, which can be obtained from the strong openness theorem 
(see \cite{GZ15}, \cite{Lem14}, \cite{Hie14}) and property~(e).

\begin{lemma}\label{mul}
For a quasi-psh function $\varphi$, we have 
$\mathcal{I}(he^{-\varphi})=\mathcal{I}(h_{\varepsilon}e^{-\varphi})$. 
In particular, we obtain the last conclusion of property~(c). 
\end{lemma}
\begin{proof}
We have the inclusion $\mathcal{I}(he^{-\varphi}) \subset \mathcal{I}(h_{\varepsilon}e^{-\varphi})$ by $h_{\varepsilon} \leq h$. 
To get the converse inclusion, 
we consider a local holomorphic function $g$ such that $|g|^2e^{-\varphi-\phi_{\varepsilon} }$ is integrable, 
where $\phi_{\varepsilon}$ (resp.\ $\phi$) is a local weight of $h_{\varepsilon}$ (resp.\ $h$). 
Then H$\rm{\ddot{o}}$lder's inequality yields 
\begin{align*}
\int |g|^2e^{-\phi-\varphi}
&= \int |g|^2 e^{-\varphi -\phi_{\varepsilon}}e^{-\phi+\phi_{\varepsilon}} \\
&\leq \Big( \int |g|^{2p} e^{-p(\varphi +  \phi_{\varepsilon}) }\Big)^{1/p} \cdot 
\Big( \int e^{-q(\phi-  \phi_{\varepsilon}) }\Big)^{1/q}, 
\end{align*}
where $p$, $q$ are real numbers such that 
$1/p+ 1/q=1$ and $p>1$.  
By the strong openness theorem, 
the function $|g|^{2p} e^{-p(\varphi +  \phi_{\varepsilon})}$ is integrable 
when $p$ is sufficiently close to one. 
On the other hand, we have 
$$
\int e^{-q(\phi -  \phi_{\varepsilon}) }-  1 
=\int e^{q\phi_{\varepsilon}}\big( e^{-q\phi} - e^{-q\phi_{\varepsilon}} \big)
\leq \sup e^{q\phi_{\varepsilon}} \int \big( e^{-q\phi} - e^{-q\phi_{\varepsilon}} \big).  
$$
The right hand side is finite for a sufficiently small $\varepsilon$ by property~(e). 
\end{proof}
This concludes the proof of Proposition \ref{appro}. 
\end{proof}

From now on, we proceed to prove Theorem \ref{main-theorem} by using
Proposition \ref{appro}. In the same way as in \cite[Section 5]{FM16}, 
one constructs a family of complete K\"ahler forms 
$\{\omega_{\varepsilon,\delta}\}_{0< \delta\ll 1 }$
on $Y_{\varepsilon}:=X \setminus (Z_{\varepsilon} \cup D)$ with the following properties\,: 
\begin{itemize}
\item[(A)] $\omega_{\varepsilon, \delta}$ is a complete K\"ahler form on 
$Y_{\varepsilon}:=X \setminus (Z_{\varepsilon} \cup D)$ for every $\delta>0$.
\item[(B)] $\omega_{\varepsilon, \delta} \geq \omega $ on $Y_{\varepsilon}$ 
for every  $\delta \geq 0$. 
\item[(C)] For every point $p$ in $X$, there exists a bounded function $\Psi_{\varepsilon,\delta}$ 
on an open neighborhood $B_p$ such that 
$\omega_
{\varepsilon, \delta} =  i\ddbar \Psi_{\varepsilon,\delta} $ on $B_{p}$ and 
$\Psi_{\varepsilon,\delta}$ converges uniformly to a bounded function
that is independent of $\varepsilon$. 
\end{itemize} 

For simplicity, we put $H:=he^{-\psi} \text{ and } H_{\varepsilon}:= h_{\varepsilon}e^{-\psi}$. 
We consider a cohomology class $\beta \in H^{q}(X,K_{X} \otimes E \otimes \mathcal{I}( H))$ 
such that $\beta =0\in H^{q}(X,K_{X} \otimes E \otimes \mathcal{I}(h))$. 
By the  De Rham-Weil isomorphism 
\begin{equation*}
H^{q}(X, K_{X}\otimes E \otimes \mathcal{I}(H))
\cong 
\frac{{\rm{Ker}}\,\dbar: L^{n,q}_{(2)}(E)_{H,\omega}
\to L^{n,q+1}_{(2)}(E)_{H,\omega}}
{{\rm{Im}}\,\dbar: L^{n,q-1}_{(2)}(E)_{H,\omega}
\to L^{n,q}_{(2)}(E)_{H,\omega}}, 
\end{equation*}
the cohomology class $\beta$ can be represented 
by a $\dbar$-closed $E$-valued $(n,q)$-form $u$ with $\|u\|_{H,\omega}<\infty$
(that is, $\beta=\{u\}$). 
Here $L^{n,\bullet}_{(2)}(E)_{H,\omega}$ is 
the $L^2$-space of $E$-valued $(n,\bullet)$-forms on $X$ 
with respect to the $L^2$-norm $\|\bullet \|_{H, \omega}$ defined by 
$$
\|\bullet \|^2_{H, \omega}:= \int_{X} |\bullet |^2_{H, \omega}\, dV_{\omega}, 
$$
where $dV_{\omega}:=\omega^n /n!$ and $n:=\dim X$. 
For the $L^2$-norm $\|\bullet\|_{H_{\varepsilon}, \omega_{\varepsilon, \delta}}$ defined by 
$$
\|\bullet\|^2_{\varepsilon, \delta}:=\|\bullet\|^2_{H_{\varepsilon}, \omega_{\varepsilon, \delta}}:=
\int_{X} |\bullet |^2_{H_\varepsilon, \omega_{\varepsilon,\delta}}\, dV_{\omega_{\varepsilon,\delta}},   
$$
one can easily check that
\begin{align}\label{a1}
\|u\|_{\varepsilon, \delta} \leq \|u\|_{H, \omega_{\varepsilon, \delta}} \leq \|u\|_{H, \omega} <\infty.  
\end{align}
Indeed, the first inequality is obtained from property $(b)$,   
and the second inequality is obtained from 
property $(B)$ for $\omega_{\varepsilon, \delta}$ (for example see \cite[Lemma 2,4]{FM16}). 
In particular, we see that $u$ belongs to the $L^{2}$-space 
$$
L^{n,q}_{(2)}(E)_{\varepsilon, \delta}:= L^{n,q}_{(2)}(Y_{\varepsilon},E)_{H_{\varepsilon},\omega_{\varepsilon, \delta}}
$$
of $E$-valued $(n,q)$-forms on $Y_{\varepsilon}$ (not $X$) 
with respect to $\|\bullet\|_{\varepsilon, \delta}$. 
By the orthogonal decomposition 
(see for example \cite[Proposition 5.8]{Mat16a}) 
$$
L^{n,q}_{(2)}(F)_{\varepsilon, \delta}= 
{\rm{Im}}\, \dbar \,  \oplus 
\mathcal{H}^{n,q}_{\varepsilon, \delta}(F)\, \oplus 
{\rm{Im}}\, \dbar^{*}_{\varepsilon, \delta},  
$$
the $E$-valued form $u$ can be decomposed as follows\,$:$
\begin{align}\label{a2}
u=\dbar w_{\varepsilon, \delta}+u_{\varepsilon, \delta}\quad 
\text{for some } 
w_{\varepsilon, \delta} \in {\rm{Dom}\, \dbar} \subset 
L^{n,q-1}_{(2)}(E)_{\varepsilon,\delta},~~\text{and}~~
u_{\varepsilon, \delta} \in \mathcal{H}^{n,q}_{\varepsilon, \delta}(E).   
\end{align}
Here $\dbar^{*}_{\varepsilon, \delta}$ is (the maximal extension of) 
the formal adjoint of the $\dbar$-operator 
and $\mathcal{H}^{n,q}_{\varepsilon, \delta}(E)$ is the space of 
harmonic forms on $Y_{\varepsilon}$, that is,     
$$
\mathcal{H}^{n,q}_{\varepsilon, \delta}(E):= 
\{ w \in L^{n,q}_{(2)}(E)_{\varepsilon, \delta} \, | \, 
\dbar w=0~~~\text{and}~~~\dbar^{*}_{\varepsilon, \delta}w=0. \}. 
$$

Proposition \ref{goal} (resp. Proposition \ref{sol}) can be proved
by the same method as in \cite[Proposition 5.4, 5.6, 5.7]{FM16} 
(resp. \cite[Proposition 5.9, 5.10]{FM16}), so we omit the proofs here. 

\begin{proposition}\label{goal}
If we have 
$$
\varliminf_{\varepsilon\to 0} \varliminf_{\delta \to 0} 
\|u_{\varepsilon, \delta}\|_{K, h_{\varepsilon}, \omega_{\varepsilon,\delta}} =0, 
$$
for every relatively compact set $K \Subset X \setminus D$, 
then the cohomology class $\beta$ is zero in 
$H^{q}(X,K_{X} \otimes E \otimes \mathcal{I}(H))$. 
Here $\| \bullet \|_{K, h_{\varepsilon}, \omega_{\varepsilon,\delta}}$ denotes the $L^2$-norm on $K$ 
with respect to $h_{\varepsilon}$ $($not $H_{\varepsilon}$$)$ and $\omega_{\varepsilon, \delta}$.
\end{proposition}

\begin{proposition}\label{sol}
There exists $v_{\varepsilon,\delta} \in L^{n,q-1}_{(2)}(E)_{h_{\varepsilon},\omega_{\varepsilon,\delta}}$ 
satisfying the following properties\,$:$ 
\begin{align}\label{a3}
\dbar v_{\varepsilon,\delta}=u_{\varepsilon, \delta}~~~\text{and}~~~
\varlimsup_{\delta \to 0} \| v_{\varepsilon,\delta}\|_{\varepsilon, \delta}~~
\text{is bounded by a constant independent of}~\varepsilon.
\end{align}
\end{proposition}

\begin{remark}\label{rem-sol}{\rm 
In general, we have 
$$
L^{n,\bullet}_{(2)}(E)_{\varepsilon,\delta} = 
L^{n,\bullet}_{(2)}(E)_{H_{\varepsilon}, \omega_{\varepsilon,\delta}}
\subsetneqq
L^{n,\bullet}_{(2)}(E)_{h_{\varepsilon},\omega_{\varepsilon,\delta}}, 
$$
and thus $v_{\varepsilon,\delta}$ may not be $L^2$-integrable with respect to $H_{\varepsilon}$. }
\end{remark}

For the above solution $v_{\varepsilon,\delta}$ of the $\dbar$-equation, 
by using the density lemma, 
we can take a family of smooth $E$-valued forms $\{ v_{\varepsilon,\delta, k} \}_{k=1}^{\infty}$ 
with the following properties\,$:$ 
\begin{align}\label{a4}
v_{\varepsilon,\delta, k} \to v_{\varepsilon,\delta} \text{\quad and \quad} 
\dbar v_{\varepsilon,\delta, k} \to \dbar v_{\varepsilon,\delta}=u_{\varepsilon,\delta} \text{ in } 
L^{n,\bullet}_{(2)}(E)_{h_{\varepsilon},\omega_{\varepsilon,\delta}}. 
\end{align}
Now we consider the level set 
$X_{c}:=\{x \in X \,|\, -|s|_{b}^{2}< c \} \Subset X \setminus D$ for a negative number $c$. 
The set of the critical values of $|s|_{b}^{2}$ is of Lebesgue measure zero 
from Sard's theorem. 
Hence, for a given relatively compact $K \Subset X \setminus D$, 
we can choose $-1 \ll c < 0$ such that 
$$
K \Subset X_{c}:=\{x \in X \,|\, -|s|_{b}^{2}< c \} \text{\quad and \quad} 
d|s|_{b}^{2}  \not = 0~~\text{at every point in}~\partial X_{c}. 
$$
Then, by \cite[Proposition 2.5, Remark 2.6]{Mat16b} (see also \cite[(1.3.2) Proposition]{FK}), we obtain  
\begin{align}\label{key}
\lla \dbar v_{\varepsilon, \delta, k}, u_{\varepsilon,\delta}
\rra_{X_{d}, h_{\varepsilon}, \omega_{\varepsilon,\delta}}=
\lla v_{\varepsilon, \delta, k}, \dbar^{*}_{h_\varepsilon, \omega_{\varepsilon, \delta}} u_{\varepsilon,\delta} 
\rra_{X_{d}, h_{\varepsilon}, \omega_{\varepsilon,\delta}} - 
\la v_{\varepsilon, \delta, k}, (\dbar |s|_{b}^2 )^{*} u_{\varepsilon,\delta} 
\ra_{\partial X_{d}, h_{\varepsilon}, \omega_{\varepsilon,\delta}}
\end{align}
for almost all $d \in{}]c-a, c+a[$, 
where $a$ is a sufficiently small positive number.  
Here $\dbar^{*}_{h_{\varepsilon}, \omega_{\varepsilon, \delta}}$ is the formal adjoint of the $\dbar$-operator 
in $L^{n,\bullet}_{(2)}(E)_{h_{\varepsilon},\omega_{\varepsilon,\delta}}$ and 
$\la \bullet, \bullet \ra_{\partial X_{d}, h_{\varepsilon}, \omega_{\varepsilon,\delta}}$ 
is the inner product on the boundary $\partial X_{d}$ defined by 
$$
\la a, b \ra_{\partial X_{d}, h_{\varepsilon}, \omega_{\varepsilon,\delta}}
:=\int_{\partial X_{d}} \langle a, b \rangle_{h_{\varepsilon}, \omega_{\varepsilon,\delta}}\, 
dS_{\varepsilon,\delta}, 
$$
for smooth $E$-valued forms $a$, $b$, 
where $dS_{\varepsilon,\delta}$ denotes the volume form on $\partial X_{d}$ 
defined by $dS_{\varepsilon,\delta}:= -* d|s|_{b}^2 / \big| d |s|_{b}^2 \big|_{h_{\varepsilon},\omega_{\varepsilon.\delta}}$ 
and $*$ denotes the Hodge star operator with respect to $\omega_{\varepsilon,\delta}$. 
Note that $dV_{\varepsilon,\delta}=dS_{\varepsilon,\delta} \wedge d|s|_{b}^2$. 
One can easily see that  
$$
\lim_{k \to \infty}\lla \dbar v_{\varepsilon, \delta, k}, u_{\varepsilon,\delta}
\rra_{X_{d}, h_{\varepsilon}, \omega_{\varepsilon,\delta}}
=
\lla \dbar v_{\varepsilon, \delta}, u_{\varepsilon,\delta}
\rra_{X_{d}, h_{\varepsilon}, \omega_{\varepsilon,\delta}}
=
\lla u_{\varepsilon,\delta}, u_{\varepsilon,\delta}
\rra_{X_{d}, h_{\varepsilon}, \omega_{\varepsilon,\delta}}
$$
by $(\ref{a4})$, 
and thus it is sufficient to show that 
the right hand side of equality (\ref{key}) converges to zero. 
For this purpose, 
we first prove the following proposition.

\begin{proposition}\label{harmonic}
$$
\lim_{\varepsilon\to 0} \lim_{\delta \to 0}
\| (\dbar |s|_{b}^2 )^{*} u_{\varepsilon,\delta}  \|_{\varepsilon,\delta}=0. 
$$
\end{proposition}

\begin{proof}[Proof of Proposition \ref{harmonic}]
By property~(d) and property~(B), we have 
\begin{align*}
\eta (\Theta_{h_{\varepsilon}}(E)+\gamma) -  i\ddbar \eta  
&\geq \eta^{-2} i \partial \eta \wedge \dbar \eta-\varepsilon\omega \\
&\geq \eta^{-2} i \partial \eta \wedge \dbar \eta-\varepsilon\omega_{\varepsilon,\delta}. 
\end{align*}
Since $u_{\varepsilon,\delta}$ is harmonic with respect to $H_{\varepsilon}$ and $\omega_{\varepsilon, \delta}$, 
we have $\dbar^*_{\varepsilon,\delta} u_{\varepsilon,\delta}=0$ and $\dbar u_{\varepsilon,\delta}=0$. 
Further we have $\gamma= i\ddbar \psi$ on $X \setminus D$. 
Therefore we obtain  
\begin{align*}
0\geq -\|\sqrt{\eta} D'^* u_{\varepsilon,\delta} \|^2_{\varepsilon,\delta}&= \|\sqrt{\eta} \dbar u_{\varepsilon,\delta} \|^2_{\varepsilon,\delta}+ 
\|\sqrt{\eta} \dbar^*_{\varepsilon,\delta} u_{\varepsilon,\delta} \|^2_{\varepsilon,\delta}
-\|\sqrt{\eta} D'^* u_{\varepsilon,\delta} \|^2_{\varepsilon,\delta}
\\
&=
\lla  \big(\eta \Theta_{H_{\varepsilon}} -  i\ddbar \eta \big) \Lambda u_{\varepsilon,\delta}, 
u_{\varepsilon,\delta}\rra_{\varepsilon,\delta} + 
2 {\rm{Re}}\lla \dbar \eta \wedge \dbar^*_{\varepsilon,\delta} u_{\varepsilon,\delta}, 
u_{\varepsilon,\delta} \rra_{\varepsilon,\delta}\\
&=
\lla  \big(\eta \Theta_{H_{\varepsilon}} -  i\ddbar \eta \big) \Lambda u_{\varepsilon,\delta}, 
u_{\varepsilon,\delta}\rra_{\varepsilon,\delta}\\
&\geq 
\lla  \big(\eta^{-2} i \partial \eta \wedge \dbar \eta) \Lambda u_{\varepsilon,\delta}, 
u_{\varepsilon,\delta}\rra_{\varepsilon,\delta} -\varepsilon q \|u_{\varepsilon, \delta}\|^2_{\varepsilon,\delta}.
\end{align*}
from the twisted Bochner-Kodaira-Nakano identity 
(see \cite[Lemma 2.1]{Ohs04} or \cite[Proposition 2.20. 2.21]{Fuj13}).  
On the other hand, one can easily check that 
\begin{align*}
&\lla  \big(\eta^{-2} i \partial \eta \wedge \dbar \eta) 
\Lambda u_{\varepsilon,\delta}, u_{\varepsilon,\delta}\rra_{\varepsilon,\delta}= 
\|\eta^{-1} (\dbar \eta)^{*}u_{\varepsilon,\delta} \|^2_{\varepsilon, \delta}=
\|\eta^{-1} * \partial \eta* u_{\varepsilon,\delta} \|^2_{\varepsilon, \delta}, \\ 
&\partial \eta=- \chi'(\sigma) \partial \sigma
=- \frac{\chi'(\sigma)}{(|s|_{b}^2 + c)} \partial|s|_{b}^2. 
\end{align*}
By the above arguments, we conclude that 
$$
\varepsilon q \|u_{\varepsilon, \delta}\|^2_{\varepsilon,\delta} \geq 
\| \frac{\chi'(\sigma)}{\eta(|s|_{b}^2 + c)} (\dbar |s|_{b}^2 )^{*} u_{\varepsilon,\delta}
\|^2_{\varepsilon,\delta}. 
$$
It follows that 
the left hand side converges to zero 
from $\|u\|_{H, \omega} \geq \|u_{\varepsilon, \delta}\|_{\varepsilon,\delta}$. 
Further the function $\chi'(\sigma)/\eta(|s|_{b}^2 + c)$ is bounded below 
since we have 
$$
\frac{1}{5}>|s|_{b}^2, \quad  C > \eta,   \quad \chi'(\sigma) > 1
$$ 
for some constant $C$. 
This completes the proof. 
\end{proof}

Finally we prove the following proposition by using Proposition \ref{harmonic}. 

\begin{proposition}\label{finish}
{\rm(i)}
For a relatively compact set $K \Subset X \setminus D$, 
we have 
$$\lim_{\varepsilon\to 0} \lim_{\delta \to 0}\varlimsup_{k \to 0}
\lla v_{\varepsilon, \delta, k}, \dbar^{*}_{h_\varepsilon, \delta} u_{\varepsilon,\delta} 
\rra_{K, h_{\varepsilon}, \omega_{\varepsilon,\delta}}=0.
$$
{\rm(ii)} For almost all $d \in {}]c-a, c+a[$,  
we have 
$$
\varliminf_{\varepsilon\to 0} \varliminf_{\delta \to 0}\varliminf_{k \to 0}
\la v_{\varepsilon, \delta, k}, (\dbar |s|_{b}^2 )^{*} u_{\varepsilon,\delta} 
\ra_{\partial X_{d}, h_{\varepsilon}, \omega_{\varepsilon,\delta}}=0. 
$$

\end{proposition}
\begin{proof}[Proof of Proposition \ref{finish}]
In general, 
we have the formula $D'_{g}u= G^{-1}\partial(Gu)$ 
for a smooth hermitian metric $g$, 
where $G$ is a local function representing $g$. 
Let $G_{\varepsilon}$ be a local function representing $h_{\varepsilon}$. 
We remark that $G_{\varepsilon}e^{-\psi}$ is a local function representing $H_{\varepsilon}$. 
By the definition of $\dbar^{*}_{\varepsilon,\delta}=\dbar^{*}_{H_\varepsilon, \omega_{\varepsilon, \delta}}$, 
we have 
$$
0=\dbar^{*}_{\varepsilon,\delta} u_{\varepsilon,\delta} =-*D'_{H_{\varepsilon}}*u_{\varepsilon,\delta}
=-*(G_{\varepsilon}e^{-\psi})^{-1}\partial(G_{\varepsilon}e^{-\psi}*u_{\varepsilon,\delta}), 
$$
and thus we obtain 
$$
\dbar^{*}_{h_\varepsilon, \omega_{\varepsilon, \delta}} u_{\varepsilon,\delta}=
-*(G_{\varepsilon})^{-1}\partial(e^{\psi} G_{\varepsilon}e^{-\psi}*u_{\varepsilon,\delta})
=-*\partial e^{\psi} *u_{\varepsilon,\delta} e^{-\psi}
=- (\dbar |s|_{b}^2)^* u_{\varepsilon,\delta} e^{-\psi}. 
$$
Now we have 
\begin{align*}\lim_{k \to \infty}
\big| \lla v_{\varepsilon, \delta, k}, \dbar^{*}_{h_\varepsilon, \omega_{\varepsilon, \delta}} u_{\varepsilon,\delta} 
\rra_{K, h_{\varepsilon}, \omega_{\varepsilon,\delta}} \big| 
&\leq \lim_{k \to \infty} \| v_{\varepsilon, \delta, k}\|_{K, h_{\varepsilon}, \omega_{\varepsilon,\delta}}
\| \dbar^{*}_{h_\varepsilon, \omega_{\varepsilon, \delta}} u_{\varepsilon,\delta} \|_{K, h_{\varepsilon}, \omega_{\varepsilon,\delta}} \\
&= \| v_{\varepsilon, \delta}\|_{K, h_{\varepsilon}, \omega_{\varepsilon,\delta}}
\| \dbar^{*}_{h_\varepsilon, \omega_{\varepsilon, \delta}} u_{\varepsilon,\delta} \|_{K, h_{\varepsilon}, \omega_{\varepsilon,\delta}}. 
\end{align*}
Since 
$\varlimsup_{\delta \to 0} \| v_{\varepsilon, \delta}\|_{K, h_{\varepsilon}, \omega_{\varepsilon,\delta}}$
can be bounded by a constant that is independent of $\varepsilon$, 
it is sufficient to show that
$\lim_{\varepsilon\to 0}\lim_{\delta \to 0}\| \dbar^{*}_{h_\varepsilon, \omega_{\varepsilon, \delta}} u_{\varepsilon,\delta} \|_{K, h_{\varepsilon}, \omega_{\varepsilon,\delta}}=0$. 
We have $ e^{-\psi/2}=1/|s|_{b} < C_{K}$ on $K$ for some constant $C_{K}>0$,
since $K$ is a relatively compact set in $X \setminus D$. 
Hence we see that 
\begin{align*}
\| \dbar^{*}_{h_\varepsilon, \omega_{\varepsilon, \delta}} u_{\varepsilon,\delta} \|_{K, h_{\varepsilon}, \omega_{\varepsilon,\delta}}=
\| - (\dbar |s|_{b}^2)^* u_{\varepsilon,\delta} e^{-\psi} \|_{K, h_{\varepsilon}, \omega_{\varepsilon,\delta}}
\leq 
C_{K}\|  (\dbar |s|_{b}^2)^* u_{\varepsilon,\delta} \|_{K,\varepsilon, \delta}. 
\end{align*}
We obtain the first statement (i)
since the right hand side converges to zero by Proposition \ref{harmonic}. 
\vskip8pt

Now we prove statement (ii). By the Cauchy-Schwarz inequality, we have 
$$
\big| \la v_{\varepsilon, \delta, k}, (\dbar |s|_{b}^2 )^{*} u_{\varepsilon,\delta} 
\ra_{\partial X_{d}, h_{\varepsilon}, \omega_{\varepsilon,\delta}} \big|^{2}
\leq 
\la v_{\varepsilon, \delta, k}, v_{\varepsilon, \delta, k}
\ra_{\partial X_{d}, h_{\varepsilon}, \omega_{\varepsilon,\delta}} 
\la (\dbar |s|_{b}^2 )^{*} u_{\varepsilon,\delta}, (\dbar |s|_{b}^2 )^{*} u_{\varepsilon,\delta}
\ra_{\partial X_{d}, h_{\varepsilon}, \omega_{\varepsilon,\delta}}.  
$$
By Fubini's theorem, we obtain 
$$
\int_{d \in{}]c-a, c+a[} \la v_{\varepsilon, \delta, k}, v_{\varepsilon, \delta, k}
\ra_{\partial X_{d}, h_{\varepsilon}, \omega_{\varepsilon,\delta}} dS_{\varepsilon,\delta}
=
\int_{c-a< -|s|_{b}^2 <c+a} 
| v_{\varepsilon, \delta, k}|^2_{h_{\varepsilon}, \omega_{\varepsilon,\delta}} dV_{\varepsilon,\delta}
\leq \|v_{\varepsilon,\delta}\|^2_{h_{\varepsilon},\delta}. 
$$
By Fatou's lemma, we see that 
$$
\int_{d \in{}]c-a, c+a[} 
\varliminf_{\varepsilon\to 0} \varliminf_{\delta \to 0} \varliminf_{k \to \infty}
\la v_{\varepsilon, \delta, k}, v_{\varepsilon, \delta, k}
\ra_{\partial X_{d}, h_{\varepsilon}, \omega_{\varepsilon,\delta}} dS_{\varepsilon,\delta}
\leq \varliminf_{\varepsilon\to 0} \varliminf_{\delta \to 0} 
\|v_{\varepsilon,\delta}\|^2_{h_{\varepsilon},\delta}< \infty. 
$$
Therefore the integrand of the left hand side is finite 
for almost all $d \in (c-a, c+a)$. 
On the other hand, 
by the same argument, we see that 
$$
\int_{d \in{}]c-a, c+a[} 
\varliminf_{\varepsilon\to 0} \varliminf_{\delta \to 0} 
\la (\dbar |s|_{b}^2 )^{*} u_{\varepsilon,\delta}, (\dbar |s|_{b}^2 )^{*} u_{\varepsilon,\delta}
\ra_{\partial X_{d}, h_{\varepsilon}, \omega_{\varepsilon,\delta}} dS_{\varepsilon,\delta}
\leq \varliminf_{\varepsilon\to 0} \varliminf_{\delta \to 0} 
\| (\dbar |s|_{b}^2 )^{*} u_{\varepsilon,\delta} \|_{h_{\varepsilon}, \omega_{\varepsilon,\delta}}=0. 
$$
Therefore the integrand of the left hand side is zero 
for almost all $d \in (c-a, c+a)$. 
This completes the proof. 
\end{proof}

Theorem \ref{main-theorem} is now a consequence of
Proposition \ref{goal}, Proposition \ref{finish}, and equation (\ref{key}). 
\end{proof}

\begin{remark}\label{inj-holconvex}{\rm 
In the case of a holomorphically convex manifold, a proof based on
injectivity theorems can be obtained by a slight modification of the above
proof. The only problem is that an $E$-valued differential form $u$
representing a given cohomology class is not necessarily \hbox{$L^2$-integrable}
but just locally $L^2$-integrable. Since $X$ admits a holomorphic map
$\pi: X \to S$ to a Stein space~$S$, the form $u$ is $L^{2}$-integrable
with respect to the metric $he^{-\psi}e^{-\Phi}$ for
a suitable psh exhaustion function $\Phi$ on $X$. 
Then it is not hard to check that our arguments still work by replacing
$h$ with~$he^{-\Phi}$.} 
\end{remark}

\begin{remark}{\rm It would be interesting to know whether the hypothesis that
$\psi$ has analytic singularities is really needed. The main statement still
makes sense when $\psi$ has arbitrary analytic singularities, and one may 
thus guess that the result can be extended by performing a 
further regularization of $\psi$.}
\end{remark}

\vskip18pt


\bigskip
\end{document}